\documentclass{amsart}

\usepackage{mathtools}
\usepackage{amssymb}
\usepackage{amsthm}
\usepackage{mathrsfs}
\usepackage{stmaryrd}
\usepackage{bbm}
\usepackage{upgreek}
\usepackage{tikz-cd}
\usepackage{chngcntr}
\usepackage{hyperref}
\usepackage{enumitem}
\usepackage[margin=3cm]{geometry}

\newcommand{\eql}{\alph*}

\counterwithin*{equation}{section}

\newtheorem{theorem}{Theorem}
\newtheorem{lemma}[theorem]{Lemma}

\newtheorem{corollary}[theorem]{Corollary}
\newtheorem{remark}[theorem]{Remark}
\theoremstyle{definition}
\newtheorem{definition}[theorem]{Definition}

\newcommand{\catname}[1]{\mathsf{#1}}
\DeclareMathOperator{\Fam}{\catname{Fam}}

\newcommand{\Cat}{\catname{Cat}}
\newcommand{\CAT}{\catname{CAT}}
\newcommand{\Ord}{\catname{Ord}}
\newcommand{\Set}{\catname{Set}}
\newcommand{\Top}{\catname{Top}}
\newcommand{\Alg}{\catname{Alg}}

\newcommand{\CoKl}{\catname{CoKl}}
\newcommand{\Desc}{\catname{Desc}}
\newcommand{\FF}{\catname{FF}}
\newcommand{\Cla}{\mathscr{C}}
\newcommand{\Clc}{\mathscr{D}}

\newcommand{\subname}[1]{\mathsf{#1}}
\newcommand{\op}{\subname{op}}
\newcommand{\co}{\subname{co}}

\newcommand{\opname}[1]{\mathsf{#1}}

\newcommand{\id}{\opname{id}}
\DeclareMathOperator{\ob}{\opname{ob}}

\newcommand{\adj}{\dashv}

\newcommand{\dash}{\text{-}}

\newcommand{\comma}{\downarrow}
\newcommand{\lcomma}{\Downarrow} 
\newcommand{\iso}{\cong}
\newcommand{\eqv}{\simeq}

\newcommand{\Ff}{\mathcal F}

\newcommand{\Hh}{\mathcal H}

\renewcommand{\epsilon}{\varepsilon}

\newcommand{\bicat}[1]{\mathbb{#1}}
\newcommand{\cat}[1]{\mathcal{#1}}

\newcommand{\init}{\mathsf{0}}
\newcommand{\term}{\mathsf{1}}

\begin{document}

\title{Functors Preserving Effective Descent Morphisms}

\date{\today} 

\author[F. Lucatelli Nunes]{Fernando Lucatelli Nunes}

\author[R. Prezado]{Rui Prezado}

\address[1,2]{University of Coimbra}
\address[1]{Utrecht University}
\address[2]{University of Aveiro}

\email[1]{f.lucatellinunes@uu.nl}
\email[2]{ruiprezado@ua.pt}

 

\begin{abstract}
Effective descent morphisms, originally defined in Grothendieck descent
theory, form a class of special morphisms within a category. Essentially, an
effective descent morphism enables bundles over its codomain to be fully
described as bundles over its domain endowed with additional algebraic structure, called descent data. Like the study of epimorphisms, studying
effective descent morphisms is interesting in its own right, providing deeper
insights into the category under consideration. Moreover, studying these
morphisms is part of the foundations of several applications of descent theory,
notably including Janelidze-Galois theory, also known as categorical Galois
theory. 

Traditionally, the study of effective descent morphisms has focused on
investigating and exploiting the reflection properties of certain functors. In
contrast, we introduce a novel approach by establishing general results on the
preservation of effective descent morphisms. We demonstrate that these
preservation results enhance the toolkit for studying such morphisms, by
observing that all Grothendieck (op)fibrations satisfying mild conditions fit
our framework. To illustrate these findings, we provide several examples of
Grothendieck (op)fibrations that preserve effective descent morphisms, including
topological functors and other forgetful functors of significant interest in the
literature. 
\end{abstract}

\maketitle

\section{Introduction}
  \label{sect:introduction}
  Given an abstract notion of categories of bundles over each object $e$ of a
category $\cat A $ -- usually provided by a fibration $\Ff$ over $\cat A$   --
the \textit{effective descent morphisms} are those morphisms \( p \colon e \to b
\) for which the corresponding change-of-base functor witnesses the bundles over
\( b \) as bundles over \(e \) with additional algebraic structure -- which we call \textit{descent data}.

For any such notion of bundle, studying and characterizing effective descent
morphisms is compelling in its own right, in addition to being a fundamental
aspect of Grothendieck descent theory and its applications, including
Janelidze-Galois theory~\cite{BJ01, JSS93}. For general and comprehensive introductions to descent theory and effective descent morphisms, we refer the reader to \cite{JT97, Luc22, Pav91, PreTh}.

For the sake of clarity and to present our main motivating examples, we decided
to focus on the quintessential context where morphisms with codomain $e$ are the
bundles over $e$; more precisely, the categories of bundles over \( e \) are
just the comma categories \( \cat A \comma e \). Although we claim that our work
can be applied to more general notions of bundles -- provided by more
general settings following \cite{Pav91}, or, more generally, by
pseudocosimplicial categories in the spirit of \cite{Luc18a} --  our results are
already noteworthy and relevant within our scope.

We study a slight generalization of the classical notion of effective descent morphisms with respect to the \textit{codomain fibration} relaxing the traditional requirement that the category must have all pullbacks.
We outline our setting below.

Let \( \cat A \) be a category and 
 \( p \colon e \to b \) a morphism in $\cat A$. 
We assume that morphisms along $p$ exist. We, then, consider the equivalence relation \( \mathtt{Eq}\left( p\right)  \), given by Diagram \eqref{eq:diagram-equivalence-relation}, induced by the \textit{kernel pair}  of \( p \).

\begin{equation}\label{eq:diagram-equivalence-relation}
  \begin{tikzcd}
e 
          \ar[r]
      & e \times_b e \ar[l,shift left=2mm]
                     \ar[l,shift right=2mm]
      & e \times_b e \times_b e  \ar[l]
                     \ar[l,shift left=2mm]
                     \ar[l,shift right=2mm]
  \end{tikzcd}
\end{equation}
Moreover, we consider the usual \textit{change-of-base functor}
\begin{equation} 
  \label{eq:change-of-base}
  p^* \colon \cat A \comma b \to \cat A \comma e 
\end{equation}
defined by taking pullbacks along \( p \). Under our assumptions, we can similarly define change-of-base functors for each of the morphisms appearing in Diagram \eqref{eq:diagram-equivalence-relation}. In other words, we get the (truncated) pseudocosimplicial category
\begin{equation}
\label{eq:pseudo-cosimplicial-category}
  \begin{tikzcd}
    \cat A \comma e \ar[r,shift left=2mm]
                        \ar[r,shift right=2mm]
      & \cat A \comma (e \times_b e) \ar[l] \ar[r,shift left=2mm]
                        \ar[r] \ar[r,shift right=2mm]
      & \cat A \comma (e \times_b e \times_b e) 
  \end{tikzcd}
\end{equation}
by taking the respective change-of-base functors. The conical bilimit of \eqref{eq:pseudo-cosimplicial-category}  -- called \textit{descent category} or \textit{category of
descent data} -- induces a factorization 
\begin{equation}
  \label{Eq:descent-factorization}
  \begin{tikzcd}
    \cat A \comma b \ar[rr, "p^*"] \ar[rd,"\mathcal K_p",swap]
      && \cat A \comma e \\
    & \Desc(p) \ar[ru]
  \end{tikzcd}
\end{equation}
called the \textit{descent factorization of $p$} -- see, for instance,
\cite[Sections 3 and 4]{Luc21} or \cite{Luc18b, Luc18a} for this approach, and  \cite{Luc22, Str04} for details on the definitions of \textit{descent category} within our setting. We say that \textit{$p$ is of
effective descent} if the comparison functor $\mathcal K_p$ is an equivalence. 

By the so-called Bénabou-Roubaud theorem~\cite{BR70}, even in this generalized
framework (see \cite[Theorem~8.5]{Luc18a} for the generalized version), the Diagram \eqref{Eq:descent-factorization}
coincides with the Eilenberg-Moore factorization 
\begin{equation*}
  \label{Eq:Eilenberg-factorization}
  \begin{tikzcd}
    \cat A \comma b \ar[rr, "p^*"] \ar[rd]
      && \cat A \comma e \\
    & T^p\dash\Alg \ar[ru]
  \end{tikzcd}
\end{equation*}
of the adjunction \( p_! \adj p^* \), up to a suitable pseudonatural equivalence.
For a detailed exposition of the considerations and perspectives outlined above, we refer the interested readers to \cite{Luc18a, Luc21, Luc22}.

Henceforth, we only need to keep in mind
the following consequence, which we \textit{adopt} herein as the definition of
effective descent morphisms:
\begin{equation*}
  \textit{$p$ is of effective descent morphism if and only if \( p^* \) is
  monadic.}
\end{equation*}

Characterization results are available for well-behaved classes of categories, such as \textit{Barr-exact} categories and \textit{locally cartesian closed} categories, where the effective descent morphisms are precisely the regular epimorphisms~\cite{JST04, JT94, Cre99}. However, in general, within the setting of our present work, characterizing effective descent morphisms often poses a challenging problem even for specific instances of categories $\cat A$ -- 
see, for example, earlier works like \cite{JT94, JST04, JS14} and more recent studies such as \cite{CJ24, PL24}. 
A celebrated illustration of this complexity is the category \( \Top \) of topological spaces and continuous maps. The characterization of effective descent morphisms in  \( \Top \), solved in \cite{RT94} and later refined in \cite{CH02, CJ11, CJ20}, demonstrates the involved nature of the problem.

In the literature, the study of effective descent
morphisms \( \cat A \) has predominantly relied on
\textit{reflection} results -- for an overview of the classical results in this area, see \cite[Section~1]{Luc18a}. 
For instance, classical reflection-based approaches have been employed in works such as \cite{CH04, CL23, RT94}. Additionally, \cite{Luc18a} introduced novel techniques to this framework, which have been applied in recent studies like \cite{CJ24, PL23, PL24}.
 
An overview of the reflection-based paradigm can be given as follows. In order to study the effective descent morphisms of a category $\cat A$, we start by looking for a category \( \cat B
\) whose effective descent morphisms are better understood and a
pullback-preserving functor \( U \colon \cat A \to \cat B \). In this setting, fall under one of the following two situations:
\begin{enumerate}[label=(\arabic*)]
  \item
    The functor \( U \colon \cat A \to \cat B \) is fully faithful. Most
    classical descent results fall under this setting, which is justified by \cite[Corollary 2.7.2]{JT94} or \cite[Corollary 9.9]{Luc18a} for the general case.
   In our setting,  Theorem~\ref{thm:obs} plays the fundamental role of this approach -- it characterizes when $p $ is of effective descent, given that  $U(p)$ is; in other words, when the property of being effective descent is reflected.

  \item
    More generally, the functor \( U \colon \cat A \to \cat B \) is a forgetful functor, in the sense that the objects of \( \cat A \) can be seen as objects of \( \cat B \) with additional structure. More precisely, we attempt to exhibit
    \( \cat A \) as a \textit{bilimit} of categories whose descent and
    effective descent morphisms are reasonably understood-- such an approach is justified by the viewpoint of \cite{Luc18a} or, more specifically, Corollary~9.5 and  \textit{ibid}. Again, these results give conditions under which the property of effective descent morphisms is reflected by $U$.
\end{enumerate}

Naturally, unless we are endowed with a preservation result, reflection results for a functor \( U \colon \cat A \to \cat B \) can only provide \textit{sufficient conditions} for a morphism to be effective for descent, and depend on the knowledge of effective descent
morphisms in \( \cat B \).  

The goal of this paper is to provide  \textit{general} preservation results for effective descent morphisms. For the purposes of the characterization problem of effective descent morphisms, we are interested in finding guiding principles
that can be applied to a wide variety of categorical settings. 

The preservation tools and techniques we present here aim to aid
our understanding of effective descent morphisms for a general class of categories, especially in the setting of the study of effective descent
morphisms in categories of generalized categorical structures -- for instance, in the context of 
\cite{PL23, PL24, CL23, PreTh, CLP24, CP24, CJ24}.

\subsection*{Outline of the paper:} 

We recall the basic notions of descent theory in
Section~\ref{sect:preliminaries}. We also take the opportunity to recall the
notion of Grothendieck (op)constructions, and we review the basic definitions needed for our work.

In Section~\ref{sect:descent}, we provide an abstract result on preservation of
effective descent morphisms. More specifically, we observe that a comonad with a
cartesian counit always preserves effective descent morphisms
(Lemma~\ref{lem:comonad}). In fact, we say more: for any adjunction \( L \adj U
\) with a cartesian counit, if \( L \) reflects effective descent morphisms
cartesian counit, then \( U \) preserves them. By listing some reasonable
conditions under which \( L \) reflects effective descent morphism, we obtain
Theorem~\ref{thm:use.counit.mono}.

Section~\ref{sect:preservation} contains our main results, Theorems
\ref{thm:fib.main} and \ref{thm:opfib.main}. These provide sufficient conditions
for a(n) (op)fibration to preserve effective descent morphisms.

Employing our results on Grothendieck (op)fibrations, in
Section~\ref{sect:examples} we work out several examples of (op)fibrations that
preserve effective descent morphisms. These include 

\begin{itemize}[label=--]
  \item
    the codomain opfibration induced by any category \( \cat A \)
    (Subsection~\ref{subsect:cod}),
  \item
    the projection \( \cat A \comma \Phi \to \cat B \) of the scone of a functor
    \( \Phi \colon \cat B \to \cat A \) (Subsection~\ref{subsect:cod}),
  \item 
    the natural forgetful functor from lax comma categories to base category
    (Subsection~\ref{subsect:cokleisli}),
  \item 
    topological functors (Subsection~\ref{subsect:topol}),
  \item 
    co-Kleisli fibrations (Subsection~\ref{subsect:laxcomma}),
  \item 
    the (op-)Grothendieck constructions over the op-indexed category of lax
    comma categories and lax direct images (Subsection~\ref{subsect:laxdirect}).
\end{itemize}

In Section~\ref{sect:future-work}, we talk about future work, extending the 
framework to more general notions of bundles (provided by general indexed categories or (truncated) pseudocosimplicial categories in the spirit of \cite{Luc18a}). We show how this particular problem can be useful to further understand natural settings that arise from two-dimensional category theory.

\subsection*{Acknowledgements:}
The first author acknowledges that this research was supported by the Fields Institute for Research in Mathematical Sciences in 2023.
Both authors acknowledge partial financial support by \textit{Centro de
Matemática da Universidade de Coimbra} (CMUC), funded by the Portuguese
Government through FCT/MCTES, DOI 10.54499/UIDB/00324/2020. 
The second author acknowledges support by Fundação para a Ciência e a Tecnologia
(FCT) under CIDMA Grants UIDB/04106/2020  and UIDP/04106/2020. 

We thank George Janelidze, Manuela Sobral and Walter Tholen for their many fruitful discussions on Grothendieck descent theory. We extend special thanks to Matthijs Vákár for his insightful discussions on Grothendieck constructions and \textit{sconing}/Artin gluing, which have significantly influenced this project. We are especially grateful to Maria Manuel Clementino, whose extensive contributions -- including in-depth discussions and providing the original guiding example -- have been invaluable in shaping this work.

\section{Preliminaries}
  \label{sect:preliminaries}
  We review the fundamental notions on descent theory and (op)-Grothendieck
constructions. In particular, we rework some classical results of descent
theory to encompass settings where our category of interest is not guaranteed
to have all pullbacks, and we provide (well-known) explicit descriptions of
pullbacks in the underlying total categories of both Grothendieck and
op-Grothendieck constructions~\cite{Gra74, LV24}.
Additionally, we recall some basic properties of categories with a strict
initial object needed for our work on examples.

\subsection*{Descent theory:}

Let \( \cat A \) be a category, and \( p \colon a \to b \) be a morphism in \(
\cat A \). The direct image functor
\begin{align*}
  p_! \colon \cat A \comma a &\to \cat A \comma b \\
  f &\mapsto p \circ f
\end{align*}
has a right adjoint \( p^* \) if, and only if, \( \cat A \) has pullbacks along
\( p \). When this is case, we write \( \Desc(p) \) for the category of algebras
of the monad induced by \( p_! \adj p^* \).

We say that \( p \) is an \textit{effective descent} morphism (respectively,
\textit{descent} morphism) precisely when \( p^* \) exists and is monadic
(respectively, premonadic).

Theorem \ref{thm:obs} studies conditions for a functor to reflect effective
descent morphisms. Its traditional form assumes a stronger hypothesis, requiring
existence and preservation of all pullbacks, and is the main technique for
studying effective descent morphisms -- see, for instance, \cite{RT94, JST04,
Luc18a, PL23, CJ24}. 

\begin{theorem}
  \label{thm:obs}
  Let \( p \colon a \to b \) be a morphism in \( \cat A \) such that \( \cat A
  \) has pullbacks along \(p \), and let \( L \colon \cat A \to \cat C \) be a
  fully faithful functor that preserves pullbacks along \(p\). If \( L(p) \) is
  an effective descent morphism, then the following are equivalent:
  \begin{enumerate}[label=(\alph*)]
    \item
      \(p\) is an effective descent morphism,
    \item
      \label{enum:obs}
      for all pullback diagrams of the form \eqref{eq:obs},
      \begin{equation}
        \label{eq:obs}
        \begin{tikzcd}
          L(c) \ar[r] \ar[rd,"\ulcorner"{rotate=180,very near start}, phantom]
            \ar[d]
            & x \ar[d,"f"] \\
          L(a) \ar[r,"L(p)",swap] & L(b)
        \end{tikzcd}
      \end{equation}
      there exists \(d\) in \( \cat A\) such that \( x \iso L(d) \).
  \end{enumerate}
\end{theorem}

\begin{proof}
  This is an immediate consequence of \cite[Theorem 9.8]{Luc18a}; our
  hypotheses state that the change-of-base functors
  \begin{equation*}
    p^* \colon \cat A \comma b \to \cat A \comma a
    \qquad
    \big(L(p)\big)^* \colon \cat C \comma L(b) \to \cat C \comma L(a)
  \end{equation*}
  exist, and that we have an invertible 2-cell
  \begin{equation}
    \label{eq:galois}
    \begin{tikzcd}
      \cat A \comma b 
        \ar[rd,phantom,"\iso"{anchor=center}]
        \ar[r,"p^*"]
        \ar[d,"L \comma b",swap]
      & \cat A \comma a \ar[d,"L \comma a"] \\
      \cat C \comma L(b) \ar[r,"(L(p))^*",swap]
      & \cat C \comma L(a)
    \end{tikzcd}
  \end{equation}
  since \( L \) preserves pullbacks along \( p \).  Moreover, we note that the
  induced functors \( L \comma b \), \( L \comma a \) are both fully faithful.

  Thus, if \( L(p) \) is an effective descent morphism, then Theorem~9.8
  \textit{ibid} confirms that \( p \) is an effective descent morphism if and
  only if \eqref{eq:galois} is a pseudopullback diagram. This is the case if and only if \ref{enum:obs} holds.
\end{proof}

It should be noted that, under the hypothesis established above, Theorem
\ref{thm:obs} characterizes the class of morphisms for which the property
\begin{center}
    $L(p)$ effective descent morphism \( \implies \) $p$ effective descent morphism.
\end{center} 
Motivated by this result, we introduce the following.

\begin{definition}
  Let $L: \cat A \to \cat C $ be a functor. We assume that
  $\Cla$ is a class of morphisms in $\cat A$. We say that:
  
  \begin{itemize}[label=--] 
    \item 
      $L$ \textit{reflects effective descent $\Cla$-morphisms} if
      \begin{center}
        $L(p)$ effective descent morphism \( \implies \) $p$ effective descent morphism,
      \end{center} 
      provided that  \( p \) in \( \Cla \);
    \item
      $L$ \textit{preserves effective descent $\Cla$-morphisms} if
      \begin{center}
        $p$ effective descent morphism \( \implies \) $L(p)$ effective descent morphism,
      \end{center} 
      whenever \( p \) in \( \Cla \).
  \end{itemize}
  In both definitions above, whenever we do not mention the classes, we mean the class of all morphisms in their respective categories.
\end{definition}
Following the terminology established above, by Theorem \ref{thm:obs} we can get the following:

\begin{corollary}
Let $L: \cat A \to \cat C$ be a fully faithful functor.
We consider the following conditions on a class of morphisms of $\cat A$:
  \begin{enumerate}[label=(\alph*)]
    \item 
      \label{enum:pb.preserve} 
      for any morphism $p$ in $\Cla $, the category $\cat A $ has all pullbacks along $p$, and $L$ preserves them;
    \item \label{enum:obs-again}
      for any morphism $p$ in $\Cla $, $L(p)$ satisfies \ref{enum:obs} of Theorem \ref{thm:obs};
   \item  \label{enum:reflect.effect.descent} $L$ reflects effective descent 
  $\Cla$-morphisms.      
  \end{enumerate}
Let  $\Clc $ be the class of morphisms in $\cat A$ satisfying   \ref{enum:pb.preserve} and \ref{enum:obs-again}. We conclude that $\Clc $ satisfies \ref{enum:reflect.effect.descent} and, furthermore, $\Clc $
 is maximal w.r.t. satisfying the properties \ref{enum:pb.preserve} and \ref{enum:reflect.effect.descent}.
\end{corollary}

\begin{theorem}
  \label{thm:obs.r.adj}
  Let \( L \adj U \colon \cat C \to \cat A \) be an adjunction with \( L \)
  fully faithful and counit \( \epsilon \), and let \( p \colon a \to b \) be a
  morphism in \( \cat A \).

  If \( L(p) \) is an effective descent morphism, then \( p \) is an effective
  descent morphism if and only if for all pullback diagrams \eqref{eq:obs}, we
  have \( \epsilon_x \) monomorphic.
\end{theorem}

\begin{proof}
  For clarity, we recall that, under our hypothesis, if $\cat C$ has pullbacks along  \( L(p) \), then \( \cat A \) has pullbacks along $p$. Indeed, if we
  have a morphism \( f \colon c \to b \) in \( \cat A \), then we consider the
  following pullback diagram:
  \begin{equation}
    \label{eq:to.reflect}
    \begin{tikzcd}
      x \ar[r] \ar[d] 
        \ar[rd,"\ulcorner"{rotate=180},phantom,very near start]
        & L(c) \ar[d,"L(f)"] \\
      L(a) \ar[r,"L(p)",swap]
        & L(b)
    \end{tikzcd}
  \end{equation}
  By composing \( U \) with Diagram \eqref{eq:to.reflect}, and noting that \( UL
  \iso \id \), we conclude that \( \cat A \) has the pullback of \(f\) along \(
  p \).
  
  Since effective descent morphisms are stable under pullback, the induced
  morphism \( L(c) \to a \) in Diagram~\eqref{eq:obs} is effective for descent
  (and is, in particular, an extremal epimorphism). This morphism factors as
  \( \epsilon_x \circ L(q) \) for a suitable \( q \), hence, if \( \epsilon_x
  \) is a monomorphism, it must be an isomorphism. This implies \( LU(x) \iso
  x \), and Theorem \ref{thm:obs} may be applied.
\end{proof}

\subsection*{Grothendieck construction:}
Since it is a two-dimensional notion, the \textit{Grothendieck construction} over an indexed category (which strictly implies that our domain is a one-dimensional category)
has four duals (see, for instance, \cite{Lac10, Luc22, LucTh}), which 
we describe below. 

For clarity, we fix two pseudofunctors
\begin{equation*}
  \Ff \colon \cat A^\op \to \CAT,
  \qquad
  \Hh \colon \cat A \to \CAT.
\end{equation*}
We recall that the \textit{Grothendieck construction} of \( \Ff \) is a
fibration \( \int_{\cat A} \Ff \to \cat A \), whose \textit{total category} \(
\int_{\cat A} \Ff \) has set of objects \( \sum_{a \in \ob \cat A} \ob \Ff_a
\), and hom-sets 
\begin{equation}
  \label{eq:gc.homsets}
  \int_{\cat A} \Ff \big( (a,x), (b,y) \big)
    = \sum_{f \in \cat A(a,b)} \Ff_a \big( x, \Ff_f(y)\big).
\end{equation}

Dually, herein, the \textit{op-Grothendieck construction}\footnote{In the literature, op-Grothendieck construction is usually the codual of what we define herein -- see, for instance, \cite{LV23, LV24, Gra74, Joh02}.} of \( \Hh \) is an
opfibration \( \int^{\cat A} \Hh \to \cat A \), whose \textit{total category}
\( \int^{\cat A} \Hh \) has set of objects \( \sum_{a \in \ob \cat A} \ob
\Hh_a \), and hom-sets
\begin{equation}
  \label{eq:opgc.homsets}
  \int^{\cat A} \Hh \big( (a,x), (b,y) \big)
    = \sum_{f \in \cat A(a,b)} \Hh_a \big( \Hh_f(x), y\big).
\end{equation}

The total categories of the codual of the op-Grothendieck and Grothendieck constructions presented above are, respectively, given by:
\begin{center} 
  \( \displaystyle\left( \int^{\cat A} \Hh \right) ^\op  \)
  and 
  \( \displaystyle\left( \int_{\cat A} \Ff \right) ^\op  \) .
\end{center} 

\begin{remark}
  \label{rem:gc}
  If \( \Ff_a = \Hh_a \) for all objects \( a \) in \( \cat A \) and \( \Hh_f
  \adj \Ff_f \), then it is immediate that the Grothendieck construction of \(
  \Ff \) is isomorphic to the op-Grothendieck construction of \( \Hh \) -- so
  that \( \int_{\cat A} \Ff \to \cat A \) is a bifibration.

  Moreover, the op-Grothendieck construction of \( \Hh \) can be obtained from
  the Grothendieck construction, via
  \begin{equation*}
    \int^{\cat A} \Hh 
      = \Big( \int_{\cat A^\op} (-)^\op \circ \Hh^\co \Big)^\op,
  \end{equation*}
  where \( (-)^\op \colon \CAT^\co \to \CAT \) is the dualization 2-functor,
  and \( \Hh^\co \colon \cat A \to \CAT^\co \) is the codual of \( \Hh \).

\end{remark}

The following results characterizing pullbacks in \( \int_{\cat A} \Ff \) and
\( \int^{\cat A} \Hh \) are special cases of familiar characterizations for
limits in (op-)Grothedieck constructions, appearing in the literature at least
since \cite{Gra66}. See, for instance, \cite[Lemma 1, Lemma 2]{LV24} for the
statements and proofs.

\begin{lemma}
  \label{lem:gc.pb.criteria}
  Let \( \Ff \colon \cat A^\op \to \CAT \) be a pseudofunctor. If we have a
  commutative square 
  \begin{equation}
    \label{eq:comm.sq}
    \begin{tikzcd}
      (a,w) \ar[r,"{(p,\pi)}"] 
            \ar[d,"{(q,\chi)}",swap]
      & (b,x) \ar[d,"{(f,\phi)}"] \\
      (c,y) \ar[r,"{(g,\psi)}",swap]
      & (d,z)
    \end{tikzcd}
  \end{equation}
  in \( \int_{\cat A} \Ff \) such that
  \begin{equation*}
    \begin{tikzcd}
      a \ar[r,"p"] 
            \ar[d,"q",swap]
      & b \ar[d,"f"] \\
      c \ar[r,"g",swap]
      & d
    \end{tikzcd}
  \end{equation*}
  is a pullback diagram, then the following are equivalent:
  \begin{enumerate}[label=(\alph*)]
    \item
      Diagram \eqref{eq:comm.sq} is a pullback diagram.
    \item
      For every morphism \( k \) in \( \cat A \comma a \), the following
      diagram
      \begin{equation*}
        \begin{tikzcd}
          \Ff_k(w) \ar[r,"\Ff_k(\pi)"] 
            \ar[d,"\Ff_k(\chi)",swap]
          & \Ff_k\Ff_p(x) \ar[d,"\Ff_k\Ff_p(\phi)"] \\
          \Ff_k\Ff_q(y) \ar[r,"\Ff_k\Ff_q(\psi)",swap]
          & \Ff_k\Ff_q\Ff_g(z) \iso \Ff_k\Ff_p\Ff_f(z)
        \end{tikzcd}
      \end{equation*}
      is a pullback square.
  \end{enumerate}
\end{lemma}

\begin{lemma}
  \label{lem:opgc.pb.criteria}
  Let \( \Hh \colon \cat A \to \CAT \) be a pseudofunctor. If we have a
  commutative square
  \begin{equation}
    \label{eq:comm.sq.1}
    \begin{tikzcd}
      (a,w) \ar[r,"{(p,\pi)}"] 
            \ar[d,"{(q,\chi)}",swap]
      & (b,x) \ar[d,"{(f,\phi)}"] \\
      (c,y) \ar[r,"{(g,\psi)}",swap]
      & (d,z)
    \end{tikzcd}
  \end{equation}
  in \( \int^{\cat A} \Hh \) such that
  \begin{equation*}
    \begin{tikzcd}
      a \ar[r,"p"] 
            \ar[d,"q",swap]
      & b \ar[d,"f"] \\
      c \ar[r,"g",swap]
      & d
    \end{tikzcd}
  \end{equation*}
  is a pullback diagram, then the following are equivalent:
  \begin{enumerate}[label=(\eql)]
    \item
      Diagram \eqref{eq:comm.sq.1} is a pullback.
    \item
      For every pair of morphisms \( \sigma \colon \Hh_p(v) \to x \), \( \tau
      \colon \Hh_q(v) \to y \) such that the diagram
      \begin{equation*}
        \begin{tikzcd}
          \Hh_g\Hh_q(v) \iso \Hh_f\Hh_p(v)
                \ar[r,"\Hh_f(\sigma)"] 
                \ar[d,"\Hh_g(\tau)",swap]
          & \Hh_f(x) \ar[d,"\phi"] \\
          \Hh_g(y) \ar[r,"\psi",swap]
          & z
        \end{tikzcd}
      \end{equation*}
      in \( \Hh_d \) commutes, there exists a unique morphism \( \theta \colon
      v \to w \) such that \( \sigma = \pi \circ \Hh_p(\theta) \) and \( \tau
      = \chi \circ \Hh_q(\theta) \).
  \end{enumerate}
\end{lemma}

\subsection*{Strict initial objects:}

Let \( \cat A \) be a category with an initial object \( \init \). We say that
\( \init \) is \textit{strict} if \( x \iso \init \) whenever \( \cat
A(x,\init) \) is non-empty. We note the following:

\begin{lemma}
  If \( \cat A \) has binary products of every object with \( \init \), then
  the following are equivalent:
  \begin{enumerate}[label=(\alph*)]
    \item
      \label{enum:strict.init}
      \( \init \) is a strict initial object.
    \item
      \label{enum:times.zero}
      \( a \times \init \iso \init \) for all \( a \). 
  \end{enumerate}
\end{lemma}
\begin{proof}
  The projection \( a \times \init \to \init \) witnesses that
  \ref{enum:strict.init} \( \implies \) \ref{enum:times.zero}.

  Conversely, we assume that \ref{enum:times.zero} holds, so that \( \cat
  A(x,a) \times \cat A(x,\init) \iso \cat A(x,\init) \). If \( x \) is an
  object in \( \cat A \) such that \( \cat A(x,\init) \) is non-empty, then
  \begin{equation*}
    \cat A(x,\init) \times \cat A(x,\init) \iso \cat A(x,\init)
  \end{equation*}
  implies that \( \cat A(x,\init) \iso \term \). Hence, we must have
  \begin{equation*}
    \cat A(x,a) \iso \cat A(x,a) \times \cat A(x,\init) \iso \cat A(x,\init)
    \iso \term
  \end{equation*}
  for all \(a\) -- which entails that \( x \iso \init \), confirming
  \ref{enum:strict.init}.
\end{proof}

\begin{lemma}
  \label{lem:strict.init.pb}
  If \( \cat A \) has a strict initial object \( \init \), then
  the following commutative square 
  \begin{equation}
    \label{eq:empty.int}
    \begin{tikzcd}
      \init \ar[r] \ar[d] & \init \ar[d] \\
      a \ar[r,"p",swap] & b
    \end{tikzcd}
  \end{equation}
  is a pullback diagram for every morphism \( p \colon a \to b \) in \( \cat A
  \).
\end{lemma}

\begin{proof}
  If we have a commutative square
  \begin{equation*}
    \begin{tikzcd}
      w \ar[r] \ar[d] & \init \ar[d] \\
      a \ar[r,"p",swap] & b
    \end{tikzcd}
  \end{equation*}
  we necessarily have \( w \iso \init \), so there is nothing to be done.
\end{proof}

\section{Guiding principle}
  \label{sect:descent}
  In Section \ref{sect:preservation}, we present the foundational observations that led us to develop more structured results. Our guiding principle is that any functor preserving effective descent morphisms induces another functor with the same property, provided it factors through a functor that reflects effective descent morphisms. Given the availability of several techniques in the literature for obtaining functors that reflect morphisms (see, for instance, \cite{RT94, JST04, Luc18a}), we consider this principle to be fruitful. We begin by stating a basic naive general result, followed by more structured versions derived through our comonadic approach.

\begin{lemma}
  \label{lem:guiding}
  Let \eqref{eq:provocar-o-Rui} be a diagram of functors, and let $\Cla$
be a class of morphisms in $\cat A$.  
  \begin{equation}\label{eq:provocar-o-Rui}
    \begin{tikzcd}
      \cat A \ar[r,"U"] 
         & \cat B \ar[r,"L"]
         & \cat C
    \end{tikzcd}
  \end{equation}
We consider the class $U\left( \Cla\right) = \Clc $  of morphisms in the image of $\Cla$ by $U$. In this setting, if 
\begin{itemize} 
\item \( LU \) preserves effective descent $\Cla$-morphisms,
\item and \( L \) reflects effective descent $\Clc$-morphisms,
\end{itemize} 
then \( U \) preserves effective descent $\Cla$-morphisms.
\end{lemma}

\subsection*{Preservation via comonads}

In this subsection, we consider an adjunction
\begin{equation*}
  \begin{tikzcd}
    \cat A \ar[r,bend left,"L"] 
           \ar[r,phantom,"\adj"{rotate=-90,anchor=center}]
      & \cat C \ar[l,bend left,"U"]
  \end{tikzcd}
\end{equation*}
whose counit is denoted by \( \epsilon \colon LU \to \id \).

\begin{lemma}
  \label{lem:comonad}
  If the counit \( \epsilon \colon LU \to \id \) is a cartesian natural
  transformation, then \( LU \) preserves effective descent morphisms.
\end{lemma}

\begin{proof}
  By definition, we have a pullback square
  \begin{equation*}
    \begin{tikzcd}
      LU(x) \ar[d,"\epsilon_x",swap]
          \ar[r,"LU(f)"]
          \ar[rd,"\ulcorner"{very near start,rotate=180},phantom]
      & LU(y) \ar[d,"\epsilon_y"] \\
      x \ar[r,"f",swap] & y
    \end{tikzcd}
  \end{equation*}
  for every morphism \( f \colon x \to y \), and since effective descent
  morphisms are stable under pullback, \( LU(f) \) is an effective descent
  morphism whenever \(f\) is. 
\end{proof}

From the preliminaries and Lemma~\ref{lem:comonad}, we conclude that:

\begin{theorem}
  \label{thm:use.counit.mono}
  We assume that \( \epsilon \) is a cartesian natural transformation that is also componentwise a monomorphism.
  
  If \( L \) is fully faithful, then \( U \) preserves effective descent $\Cla$-morphisms, where $\Cla $ consists of the morphisms $p$ such that $L$ preserves pullbacks along $U(p) $.
\end{theorem}
\begin{proof}
  Let \( p \) be an effective descent morphism in \( \cat C \) that is in the class $\Cla$. 
  
  By Lemma
  \ref{lem:comonad}, \( LU(p) \) is an effective descent morphism. Now our
  result follows from Theorem~\ref{thm:obs.r.adj}: since \( \epsilon \) is a
  componentwise monomorphism, \( L \) is fully faithful, and \( L \) preserves
  pullbacks along \( U(p) \), we conclude that \( U(p) \) is an effective
  descent morphism.
\end{proof}

\section{Preservation results}
  \label{sect:preservation}
  Many examples of (Grothendieck) fibrations and opfibrations are under the
conditions of Theorem~\ref{thm:use.counit.mono}, and therefore preserve
effective descent morphisms -- these observations are the main result of this
note, which are studied in this section. 

This is split into two subsections, which respectively contain our results for
fibrations and opfibrations. These results are not dual to each other in any
sense -- the details are more delicate in the case of opfibrations.

\subsection*{Fibrations}

We fix a pseudofunctor 
\begin{equation*}
  \Ff \colon \cat A^\op \to \CAT
\end{equation*}
and we denote its Grothendieck construction by
\begin{equation*}
  U \colon \textstyle\int_{\cat A} \Ff \to \cat A.
\end{equation*}

\begin{lemma}
  \label{lem:fib.ladj.ff}
  If \( \Ff_a \) has an initial object \( \init_a \) for every object \( a \in
  \cat A \), then \( U \) has a fully faithful left adjoint \( L \), given on
  objects by \( L(a) = (a,\init_a) \).
\end{lemma}
\begin{proof}
  We have the following natural isomorphism
  \begin{equation}
    \label{eq:left.adj.fib.0}
    \int_{\cat A} \Ff \big(L(a),(b,x)\big)
      \iso \sum_{f \in \cat A(a,b)} \Ff_a\big(\init_a,\Ff_f(x)\big)
      \iso \cat A(a,b) = \cat A\big(a,U(b,x)\big),
  \end{equation}
  exhibiting the adjunction \( L \adj U \). By taking \( x = \init_b \) in
  \eqref{eq:left.adj.fib.0} we conclude that \( L \) is fully faithful.
\end{proof}

\begin{lemma}
  \label{lem:fib.result}
  If \( \Ff_a \) has a strict initial object for all \( a \) in \( \cat A \),
  and the change-of-base functors \( \Ff_f \) preserve them for every morphism
  \( f \) in \( \cat A \), then 
  \begin{enumerate}[label=(\roman*),noitemsep]
    \item
      \label{enum:counit.cart}
      the counit \( \epsilon \) of \( L \adj U \) is a cartesian natural
      transformation,
    \item
      \label{enum:left.adj.pb}
      \( L \) preserves pullbacks,
    \item
      \label{enum:counit.mono}
      \( \epsilon \) is a componentwise monomorphism.
  \end{enumerate}
\end{lemma}

\begin{proof} 
  Via Lemma \ref{lem:gc.pb.criteria}, we may conclude that
  \ref{enum:counit.cart} holds if the following diagram is a pullback
  \begin{equation*}
    \begin{tikzcd}
      \Ff_p(0_a) \ar[r] \ar[d] 
        & \Ff_p\Ff_f(0_b) \ar[d] \\
      \Ff_p(x) \ar[r,"\Ff_p(\phi)",swap] & \Ff_p\Ff_f(y)
    \end{tikzcd}
  \end{equation*}
  for every morphism \( (f,\phi) \colon (a,x) \to (b,y) \) in \( \int_{\cat A}
  \Ff \), and every morphism \( p \colon e \to a \) in \( \cat A \). Since
  the change-of-base functors preserve initial objects, this is an immediate
  consequence of Lemma \ref{lem:strict.init.pb}.

  Let
  \begin{equation*}
    \begin{tikzcd}
      a \ar[r,"f"] \ar[d,"g",swap]
        \ar[rd,"\ulcorner"{rotate=180,very near start}, phantom]
        & b \ar[d,"h"] \\
      c \ar[r,"k",swap] & d
    \end{tikzcd}
  \end{equation*}
  be a pullback diagram. Since the change-of-base functors preserve initial
  objects, the following diagram
  \begin{equation}
    \label{eq:zero.pb}
    \begin{tikzcd}
      \Ff_p(0_a) \ar[r,"\iso"] \ar[d,"\iso",swap] 
        & \Ff_p\Ff_f(0_b) \ar[d,"\iso"] \\
      \Ff_p\Ff_g(0_c) \ar[r,"\iso"] 
        & \Ff_p\Ff_g\Ff_k(0_d) \iso \Ff_p\Ff_f\Ff_h(0_d)
    \end{tikzcd}
  \end{equation}
  is a pullback diagram for any \( p \colon e \to a \). Thus, we may apply
  Lemma \ref{lem:gc.pb.criteria} to conclude that \ref{enum:left.adj.pb}
  holds.

  To prove \ref{enum:counit.mono}, we note that \( \epsilon \) is given at \(
  (b,y) \) by the pair
  \begin{equation*}
    (\id_b, u) \colon (b,\init_b) \to (b,y),
  \end{equation*}
  where \( u \colon \init_b \to \Ff_{\id_b}(y) \) is the unique morphism. If
  we have 
  \begin{equation*}
    (\id_b,u) \circ (h,\zeta) = (\id_b,u) \circ (k,\xi)
  \end{equation*}
  for morphisms \( (h,\zeta), (k,\xi) \colon (a,x) \to (b,\init_b) \), it
  follows that \( h = k \), and since change-of-base functors preserve
  (strict) initial objects, we must have \( x \iso \init_a \), and hence \(
  \zeta = \xi \).  Thus, we conclude that \( \epsilon_{b,y} \) is a
  monomorphism.
\end{proof}

As a corollary, we obtain that:

\begin{theorem}
  \label{thm:fib.main}
  If \( \Ff_a \) has a strict initial object \( \init_a \) for each \( a \) in
  \( \cat A \), and \( \Ff_f \) preserves them for all morphisms \( f \) in \(
  \cat A \), then \( U \colon \int_{\cat A} \Ff \to \cat B \) preserves
  effective descent morphisms.
\end{theorem}

\begin{proof}  
  By Lemmas \ref{lem:fib.ladj.ff} and \ref{lem:fib.result}, \( U \) has a left
  adjoint \( L \) which enjoys all the properties needed to apply
  Theorem~\ref{thm:use.counit.mono}.
\end{proof}

\begin{corollary}
  \label{cor:pbfib.preserve}
  Let \( \Phi \colon \cat B \to \cat A \) be any functor.  If \( \Ff_a \) has
  a strict initial object \( \init_a \) for all \( a \) in \( \cat A \), and
  \( \Ff_f \) preserves initial objects for all morphisms \( f \) in \( \cat A
  \), then the Grothendieck construction
  \begin{equation*}
    \textstyle\int_{\cat B} \Ff \circ \Phi^\op \to \cat B
  \end{equation*}
  preserves effective descent morphisms as well.
\end{corollary}

\subsection*{Opfibrations}

We fix a pseudofunctor 
\begin{equation*}
  \Hh \colon \cat A \to \CAT,
\end{equation*}
and we denote its Grothendieck construction by
\begin{equation*}
  V \colon \textstyle\int^{\cat A} \Hh \to \cat A.
\end{equation*}

\begin{lemma}
  \label{lem:opfib.ladj.ff}
  If \( \Hh_a \) has an initial object \( \init_a \) for every object \( a \in
  \cat A \), and the change-of-base functors \( \Hh_f \) preserve them for
  every morphism \( f \) in \( \cat A \), then \( V \) has a fully faithful
  left adjoint \( K \), given on objects by \( K(a) = (a, \init_a) \).
\end{lemma}
\begin{proof}
  This is similar to the proof of \ref{lem:fib.ladj.ff}; since \(
  \Hh_f(\init_a) \iso \init_b \), we have
  \begin{equation}
    \label{eq:left.adj.fib.1}
    \int^{\cat A} \Hh \big(K(a),(b,x)\big)
      \iso \sum_{f \in \cat A(a,b)} \Hh_b\big(\Hh_f(\init_a),x\big)
      \iso \sum_{f \in \cat A(a,b)} \Hh_b\big(\init_b,x\big)
      \iso \cat A(a,b) = \cat A\big(a,V(b,x)\big),
  \end{equation}
  which witnesses \( K \adj V \). Taking \( x = \init_b \) in
  \eqref{eq:left.adj.fib.1} confirms that \( K \) is fully faithful. 
\end{proof}

\begin{lemma}
  \label{lem:opfib.result}
  If \( \Hh_a \) has a strict initial object \( \init_a \) for all \( a \) in
  \( \cat A \) that is preserved by the change-of-base functors \( \Hh_f \)
  for all morphisms \( f \) in \( \cat A \), then
  \begin{enumerate}[label=(\alph*)]
    \item
      \label{enum:nat.cond}
      the naturality square of \( \nu \) at a morphism \( (f,\phi) \) is a
      pullback diagram if and only if \( \Hh_f \) reflects initial objects,
    \item
      \label{enum:prsv.cond}
      the left adjoint \( K \) preseves the pullback of a pair of arrows \(
      f,g \) if and only if the change-of-base functor induced by the diagonal
      of the pullback square reflects initial objects,
    \item
      \label{enum:counit.mono.1}
      \( \nu \) is a componentwise monomorphism.
  \end{enumerate}
\end{lemma}    

\begin{proof}
  By Lemma \ref{lem:opgc.pb.criteria}, the naturality square of \( \nu \) at a
  morphism \( (f,\phi) \colon (a,x) \to (b,y) \) is a pullback diagram if and
  only if for every commutative diagram
  \begin{equation*}
    \begin{tikzcd}
      \Hh_f(v) \ar[r] \ar[d,"\Hh_f(\gamma)",swap] 
        & \init_b \ar[d] \\
      \Hh_f(x) \ar[r,swap,"\phi"]
        & y
    \end{tikzcd}
  \end{equation*}
  we have \( v \iso \init_a \). Since \( \init_b \) is strict, every such
  diagram commutes if and only if \( \Hh_f(v) \iso \init_b \), so
  \ref{enum:nat.cond} holds.

  If the following diagram 
  \begin{equation*}
    \begin{tikzcd}
      a \ar[r,"h"] \ar[d,"k",swap]
        & b \ar[d,"f"] \\
      c \ar[r,"g",swap] & d
    \end{tikzcd}
  \end{equation*}
  is a pullback in \( \cat A \), then by Lemma \ref{lem:opgc.pb.criteria}, \(
  K \) preserves this pullback if and only if for every commutative square
  \begin{equation*}
    \begin{tikzcd}
      \Hh_g\Hh_k(v) \iso \Hh_f\Hh_h(v)
        \ar[r] \ar[d]
      & \Hh_f(\init_b) \ar[d] \\
      \Hh_g(\init_c) \ar[r]
        & \init_d
    \end{tikzcd}
  \end{equation*}
  we have \( v \iso \init_a \). Since \( \init_d \) is strict, and
  change-of-base functors preserve initial objects, every such diagram
  commutes if and only if \( \Hh_f\Hh_h(v) \iso \Hh_g\Hh_k(v) \iso \init_d \),
  which confirms \ref{enum:prsv.cond}.

  \( \nu \) is given at  \( (b,y) \) by the pair
  \begin{equation*}
    (\id_b, u) \colon (b,\init_b) \to (b,y),
  \end{equation*}
  where \( u \colon \Hh_{\id_b}(\init_b) \to y \) is the unique morphism,
  since \( \Hh_{\id_b}(\init_b) \iso \init_b \). If 
  \begin{equation*}
    (\id_b,u) \circ (h,\zeta) = (\id_b,u) \circ (k,\xi)
  \end{equation*}
  for morphisms \( (h,\zeta), (k,\xi) \colon (a,x) \to (b,\init_b) \), it
  follows that \( h = k \), and since initial objects are strict, we must have
  \( \zeta = \xi \colon \Hh_f(x) \iso \init_b \). Thus, we conclude that
  \ref{enum:counit.mono.1} holds.
\end{proof}

As an immediate corollary, we obtain:
\begin{theorem}
  \label{thm:opfib.main}
  If \( \Hh_a \) has a strict initial object \( \init_a \) for all \( a \),
  which is created by every change-of-base functor, then \( V \colon \int^{\cat
  A} \Hh \to \cat A \) preserves effective descent morphisms.
\end{theorem}

\begin{corollary}
  \label{cor:pbopfib.preserve}
  Let \( \Phi \colon \cat B \to \cat A \) be any functor.  If \( \Hh_a \) has
  a strict initial object \( \init_a \) for all \( a \) in \( \cat A \), and
  \( \Hh_f \) creates them for all morphisms \( f \) in \( \cat A \), then the
  op-Grothendieck construction
  \begin{equation*}
    \textstyle\int^{\cat B} \Hh \circ \Phi \to \cat B
  \end{equation*}
  preserves effective descent morphisms as well.
\end{corollary}

\section{Examples}
  \label{sect:examples}
  
\subsection{Codomain bifibration:}
\label{subsect:cod}

Let \( \cat A \) be a category with a strict initial object, and we consider
the ``direct image'' pseudofunctor
\begin{align*}
  \Hh \colon \cat A &\to \CAT \\
  a &\mapsto \cat A \comma a \\
  p &\mapsto p_!
\end{align*}
where \( p_! \colon \cat A \comma a \to \cat A \comma b \) is given by \( f
\mapsto p \circ f \). We denote its op-Grothendieck construction by \(
\mathsf{cod} \colon \cat A^2 \to \cat A \) -- the \textit{codomain
opfibration}.

\begin{lemma}
  The codomain opfibration preserves effective descent morphisms.
\end{lemma}

\begin{proof}
  The unique morphism \( \init \to a \) is a strict initial object in
  \( \cat A \comma a \) for every \( a \), and the change-of-base functors \(
  p_! \colon \cat A \comma a \to \cat A \comma b \) create them. The result
  is therefore confirmed by Theorem~\ref{thm:opfib.main}.
\end{proof}

Let \( \Phi \colon \cat B \to \cat A \), and consider the Artin gluing (also
known as \textit{scone}) \( \cat A \comma \Phi \), together with the
projection \( U \colon \cat A \comma \Phi \to \cat B \). This fibration, called herein the \textit{scone forgetful functor},  has been extensively studied in the literature for its rich properties, particularly as a categorical viewpoint on logical relations techniques~\cite{MS93, HSV20, LV22, LV23, LV24b}. 
As an application of our framework, we find that this fibration is particularly relevant to our context in descent theory. Specifically, we obtain the following result:

\begin{lemma}
The scone forgetful functor  \( U \)  preserves effective descent morphisms.
\end{lemma}

\begin{proof}
  We apply Corollary \ref{cor:pbfib.preserve} to the composite
  \begin{equation*}
    \begin{tikzcd}
      \cat B \ar[r,"\Phi"]
        & \cat A \ar[r,"\Hh"]
        & \CAT
    \end{tikzcd}
  \end{equation*}
  and we note that \( \int^{\cat B} \Hh \circ \Phi \eqv \cat A \comma \Phi \).
\end{proof}

We also consider the Grothendieck construction \( \int_{\cat A^\op} \Hh \to
\cat A^\op \) of \( \Hh \); the objects of \( \int_{\cat A^\op} \Hh \) still
are the morphisms of \( \cat A \), while a morphism \( f \to g \) is a pair \(
(p,q) \) of morphisms such that the following square commutes:
\begin{equation*}
  \begin{tikzcd}
    a \ar[r,"p"] \ar[d,"f",swap]
      & b \ar[d,"g"] \\
    c & d \ar[l,"q"]
  \end{tikzcd}
\end{equation*}

By Theorem \ref{thm:fib.main}, we obtain:
\begin{lemma}
  The fibration \( \int_{\cat A^\op} \Hh \to \cat A^\op \) preserves effective
  descent morphisms.
\end{lemma}

\subsection{Co-Kleisli categories for writer comonads:}
\label{subsect:cokleisli}

Let \( \cat A \) be a category with binary products and a strict initial
object. For each \( a \in \cat A \), the functor \( a \times - \) underlies a
comonad on \( \cat A \). We denote its co-Kleisli category by \( \CoKl(a
\times -) \), and we recall it may be given as follows:
\begin{itemize}[label=--]
  \item
    \( \ob \CoKl(a \times -) = \ob \cat A \),
  \item
    \( \CoKl(a \times -)(x,y) = \cat A(a \times x, y) \),
  \item
    Identities are the product projections \( a \times x \to x \),
  \item
    Composition of morphisms \( f \colon a \times x \to y \), \( g \colon a
    \times y \to z \) is given by 
    \begin{equation*}
      \begin{tikzcd}
        a \times x \ar[r,"\Delta \times \id_x"]
          & a \times a \times x \ar[r,"\id_a \times f"]
          & a \times y \ar[r,"g"] 
          & z
      \end{tikzcd}
    \end{equation*}
\end{itemize}

For each morphism \( f \colon a \to b \), we let \( \Ff_f \colon
\CoKl(b \times -) \to \CoKl(a \times -) \) be given by the identity on
objects, and 
\begin{align*}
  \cat A(b \times x, y) &\to \cat A(a \times x, y) \\
  p &\mapsto p \circ (f \times \id_x)
\end{align*}
on hom-sets. This defines a pseudofunctor \( \Ff \colon \cat A^\op \to \CAT
\).

The total category of the Grothendieck construction \( U \colon \int_{\cat A}
\Ff \to \cat A \) of \( \Ff \) is the full subcategory of \( \cat A^2 \)
consisting of the product projections. To be precise, the objects of \(
\int_{\cat A} \Ff \) are pairs of objects \( (a,x) \), corresponding to
product projections \( a \times x \to a \), and a morphism \( (f,g) \colon
(a,x) \to (b,y) \) consists of morphisms \( f \colon a \times x \to b \times y
\), \( g \colon a \to b \) such that the following diagram commutes:
\begin{equation*}
  \begin{tikzcd}
    a \times x \ar[r,"f"] \ar[d]
      & b \times y \ar[d] \\
    a \ar[r,"g",swap] & b
  \end{tikzcd}
\end{equation*}

To ensure that \( \Ff \) is within the scope of Theorem \ref{thm:fib.main}, we
have

\begin{lemma}
  \label{lem:cokl.fibers}
  The co-Kleisli categories \( \CoKl(a \times -) \) have strict initial
  objects, and the change-of-base functors \( \Ff_f \) preserve them.
\end{lemma}
\begin{proof}
  We recall that the category of coalgebras for \( a \times - \) is equivalent
  to \( \cat A \comma a \) -- this allows us to view the co-Kleisli category
  \( \CoKl(a \times -) \) as the full subcategory of \( \cat A \comma a \)
  whose objects are the product projections \( d_0 \colon a \times x \to a \).

  Since \( a \times \init \iso \init \) for all \( a \), this implies that \(
  \CoKl(a \times -) \) has a strict initial object -- \( \init \) itself --
  for all \( a \).  These are preserved by the identity-on-objects
  change-of-base functors.
\end{proof}

Thus, as a corollary, we obtain:

\begin{lemma}
  The fibration \( U \colon \int_{\cat A} \Ff \to \cat A \) preserves
  effective descent morphisms.
\end{lemma}

We may also consider the op-Grothendieck construction \( \int^{\cat A^\op} \Ff
\to \cat A^\op \) of \( \Ff \). A morphism \( (a,x) \to (b,y) \) in its total
category corresponds to a pair of morphisms \( f \colon a \times x \to b
\times y \), \( g \colon b \to a \) such that the following diagram commutes:
\begin{equation*}
  \begin{tikzcd}
    a \times x \ar[r,"f"] \ar[d]
      & b \times y \ar[d] \\
    a & b \ar[l,"g"]
  \end{tikzcd}
\end{equation*}

Moreover, we note that the identity-on-objects change-of-base functors create
initial objects, whence we may apply Lemma \ref{lem:opfib.result} to conclude
that
\begin{lemma}
  The opfibration \( \int^{\cat A^\op} \Ff \to \cat A^\op \) preserves
  effective descent morphisms.
\end{lemma}

\subsection{Topological functors:}
\label{subsect:topol}

We recall from \cite{Wyl71} that any topological functor \( U \colon \cat C
\to \cat A \) (with possibly large fibers) is the Grothendieck construction of
a bifibration \( \Ff \colon \cat A^\op \to \CAT \) whose fibers \( \Ff_a \)
are (large-)complete thin categories. In particular, the fibers \( \Ff_a \)
have strict initial objects.

\begin{lemma}
  \label{lem:topol}
  If the change-of-base functors \( \Ff_f \colon \Ff_b \to \Ff_a \) preserve
  initial objects (bottom elements), then \( U \colon \cat C \to \cat A \)
  preserves effective descent morphisms.
\end{lemma}

As a consequence of Lemma~\ref{lem:topol}, the quintessential topological
functor \( \Top \to \Set \) preserves effective descent morphisms.

\subsection{Lax comma categories:}
\label{subsect:laxcomma}
Recall that, given a $2$-category \( \bicat A \) and an object \( X \) of \(
\bicat A \), the lax comma category \( \bicat A \lcomma X \) is defined as the
total category of the Grothendieck construction of the representable 2-functor
\( \bicat A(-,X) \colon \bicat A^\op \to \CAT \)\footnote{The Grothendieck
construction can be defined for pseudofunctors \( \bicat A^\op \to \CAT \) for
\( \bicat A \) a 2-category, outputting another 2-category, but for our
purposes, we are implicitly taking the domain of the 2-functor to be the
underlying category of \( \bicat A^\op \).} -- see, for instance, \cite{CL24,
CLP24, LV24}.

Unpacking the definition, the objects of \( \bicat A \lcomma X \) are
morphisms \( f \colon A \to X \) with fixed codomain \(X\), while a morphism
\( (p, \pi) \colon f \to g \) given by a 2-cell \( \pi \colon f \to g\cdot p
\), which is depicted in the following diagram:
\begin{equation*}
  \begin{tikzcd}[row sep=large]
    A \ar[rr,"p"] 
      \ar[rr,Rightarrow,shorten=8.5mm,shift right=5.5mm,"\pi"]
      \ar[rd,"f",swap]
    && B  \ar[ld, "g"] \\
    & X
  \end{tikzcd}
\end{equation*}

The composition of morphisms \( (p,\pi) \colon (A,f) \to (B,g) \) and \(
(q,\chi) \colon (B,g) \to (C,h) \) is given by the pair \( (q\cdot p, (\chi
\cdot p) \circ \pi) \), which may be obtained diagramatically by pasting the
2-cells as follows:
\begin{equation*}
  \begin{tikzcd}[row sep=large]
    A \ar[rr,"p"] 
      \ar[rr,Rightarrow,shorten=6.5mm,shift right=5.5mm,"\pi"]
      \ar[rrd,"f",swap,bend right]
    && B  
      \ar[d,"g"] 
      \ar[rr,"q"]
      \ar[rr,Rightarrow,shorten=6.5mm,shift right=5.5mm,"\chi"]
    && C  \ar[lld,"h",bend left]\\ 
    && X
  \end{tikzcd}
\end{equation*}

This work was developed from the examination of the core ideas behind the
preservation results for the forgetful functor \( \bicat A \lcomma X  \to \bicat A \), including the specific cases studied in \cite{CL23, CLP24, CP24}. 
For instance, it was shown that \( \Ord \lcomma X \to \Ord \)
preserves effective descent morphisms when \( X \) is an complete ordered set
\cite[Theorem 3.3]{CL23}, and that \( \Cat \lcomma \cat A \to \Cat \)
preserves effective descent morphisms when \( \cat A \) has pullbacks and a
strict initial object \cite[Theorem 3.7]{CLP24}. We will show that both
results can be obtained (and generalized) in the present setting.

From this point forward, we assume that \(X\) is an object of \( \bicat A \)
such that \( \bicat A(A,X) \) has a strict initial object for every object \(
A \) in \( \bicat A \), and that the change-of-base functor \( - \cdot f
\colon \bicat A(B,X) \to \bicat A(A,X) \) preserves them for every morphism \(
f \colon A \to B \) in \( \bicat A \). As an immediate consequence of Theorem
\ref{thm:fib.main}, we obtain the following result:

\begin{lemma}
  The fibration \( \bicat A \lcomma X \to \bicat A \) preserves effective
  descent morphisms.
\end{lemma}

We recover \cite[Theorem 3.3]{CL23} by taking \( \bicat A = \Ord \), and \( X
\) an ordered set with a bottom element, as well as \cite[Theorem 3.7]{CLP24}
by taking \( \bicat A = \Cat \), and \( X \) a category with a strict initial
object.

If we have a 2-functor \( \Phi \colon \bicat B \to \bicat A \) between
2-categories, we may consider the composite 2-functor \( \bicat A(\Phi(-),X)
\colon \bicat B^\op \to \CAT \) for \( X \) in \( \bicat A \), and we denote
its Grothendieck construction as \( \Phi \lcomma X \to \bicat B \). In case \(
\Phi \) is fully faithful, we denote the total category as \( \bicat B \lcomma
X \) instead.

As a consequence of Corollary \ref{cor:pbfib.preserve}, we conclude that

\begin{lemma}
  The fibration \( \Phi \lcomma X \to \bicat B \) preserves effective descent
  morphisms.
\end{lemma}

As an important example, we conclude at once that \( \Fam(X) \to \Set \)
preserves effective descent morphisms for any category \( X \) with a strict
initial object, by taking \( \Phi \) to be the ``discrete category'' functor
\( \Set \to \Cat \) -- it should be noted that \( \Fam(X) \eqv \Set \lcomma X
\). This observation is particularly helpful in the characterization of
effective descent morphisms in \( \Fam(X) \), a task which started in
\cite{Pre24}.

\subsection{Lax direct image:}
\label{subsect:laxdirect}

Let \( \bicat A \) be a 2-category with a strict initial object. We consider
the pseudofunctor 
\begin{align*}
  \Hh \colon \bicat A &\to \CAT \\ 
  x &\mapsto \bicat A \lcomma x \\ 
  p &\mapsto p_!
\end{align*}
where \( p_! \colon \bicat A \lcomma x \to \bicat A \lcomma y \) is given by
composing (whiskering) with \(p\) on morphisms (2-cells).

The total category of its op-Grothendieck construction \( \int^{\cat
A} \Hh \to \bicat A \) has
\begin{itemize}[label=--]
  \item
    objects are the morphisms of \( \bicat A \),
  \item
    morphisms \( f \to g \) consist of triples \( (p,h,\pi) \) where \( \pi
    \colon p \cdot f \to g \cdot h \) is a 2-cell, which may be depicted as
    \begin{equation*}
      \begin{tikzcd}
        a \ar[r,"h"] \ar[d,"f",swap]
          & b \ar[d,"g"] \\
        x \ar[r,"p",swap]
          \ar[ru,Rightarrow,"\pi",shorten=3mm]
        & y
      \end{tikzcd}
    \end{equation*}
\end{itemize}

Since \( \Hh_a \) has an initial object for all \( a \) in \( \bicat A \), and
the change-of-base functors create them for all \( f \colon a \to b \) in \(
\bicat A \), from Theorem \ref{thm:opfib.main} we conclude that:

\begin{lemma}
  The opfibration \( \int^{\bicat A} \Hh \to \bicat A \) preserves effective descent
  morphisms.
\end{lemma}

The total category of its Grothendieck construction \( U \colon \int_{\cat
A^\op} \Hh \to \bicat A^\op \) has
\begin{itemize}[label=--]
  \item
    objects: the morphisms of \( \bicat A \),
  \item
    morphisms \( f \to g \): triples \( (p,h,\pi) \) where \( \pi \colon g
    \to p \cdot f \cdot h \) is a 2-cell, which may be depicted as
    \begin{equation*}
      \begin{tikzcd}
        b \ar[r,"h"] \ar[d,"g",swap]
          \ar[r,"\pi",shift right=5mm,Rightarrow,shorten=2mm]
          & a \ar[d,"f"] \\
        y 
        & x \ar[l,"p"]
      \end{tikzcd}
    \end{equation*}
\end{itemize}

Since \( \Hh_a \) has an initial object for all \( a \) in \( \bicat A \), and
the change-of-base functors preserve them for all \( f \colon a \to b \) in \(
\bicat A \), we conclude that:

\begin{lemma}
  The fibration \( \int_{\cat A^\op} \Hh \to \cat A^\op \) preserves effective
  descent morphisms.
\end{lemma}

\section{Epilogue}
  \label{sect:future-work}
  As explained in the introduction, we have restricted our present work to the context where the notion of bundles over an object is given by suitable comma categories. However, there is much to be explored in generalized settings where bundles are defined by alternative notions.

For instance, we highlight the work of \cite{Sob04, LPS23}, where effective descent morphisms with respect to indexed categories of discrete fibrations and the indexed category of  split fibrations were studied.

Building upon the work mentioned above, our project on the study of descent for generalized categorical structures~\cite{PL23, LPS23, PL24, PreTh} aims to investigate effective descent morphisms with respect to generalized notions of bundles of (generalized) categorical structures, as briefly outlined below.

\subsection{Effective descent morphisms w.r.t. an orthogonal factorization systems}
Every orthogonal factorization system induces a notion of bundles in a category. We are particularly interested in the case of the category of categories $\CAT$.

For each category $\cat A$, we may consider the category $\FF(A)$
which is the full  subcategory of $\CAT \comma \cat A $ consisting of the
fully faithful functors with codomain $\cat A$. This gives us an indexed category \( \FF \colon \cat A^\op \to \CAT \) and a corresponding notion of effective descent morphism hasn't been studied in the literature.

More generally, the right class of morphisms of any orthogonal factorization system in $\CAT $ provides us with such an indexed category and hence a notion of effective descent morphism w.r.t. the given orthogonal factorization.   Notably, this was explored by \cite{Sob04} in the case of \cite{SW73}, where the right class of morphisms consists of discrete fibrations. However, following the lines above, the orthogonal factorization defined in \cite{LS22}, whose right class of morphisms consists of \textit{discrete splitting bifibrations} as introduced therein, provides us with a meaningful notion of effective descent morphisms yet to be explored.

\subsection{Effective descent morphisms w.r.t. a 2-monad}

We have a specific aim pertaining to freely generated categorical
structures~\cite{Luc17, LPV24, LV24a} and two-dimensional monad
theory -- see \cite{BKP89, Mar99,  Bou14, Luc16, Luc19} for the
general setting of $2$-dimensional monad theory.

We start by observing that any pseudomonad $\cat T $ on $\CAT$ provides us with a relevant
notion of bundles, called herein $\cat T$-bundles over a category $\cat A  $,
given by the $\cat T \left( \cat A \right) $. In this setting, we are posed with
the following question: assuming that $p \colon a \to b$ has a left (or right)
adjoint $p_*$, we end up with a relevant notion of descent factorization  
\begin{equation*}
  \begin{tikzcd}
    \cat T(a) \ar[rr, "\cat T(p_*)"] \ar[rd,"\mathcal K_p",swap]
      && \cat T(b) \\
    & \Desc_{\cat T}(p) \ar[ru]
  \end{tikzcd}
\end{equation*}
following the recipe explained in the introduction. We, then, are interested in further 
understanding when $p$ is of effective descent with respect to $\cat T$-bundles.
We are particularly interested in the case where $\cat T$ is a Kock-Zoberlein
pseudomonad, also known as lax idempotent pseudomonads~\cite{Str80, Zöb76, 
Koc95, Mar97, KL97} which encompasses most of the free completion pseudomonads~\cite{KL00, PCW00, LV24a, LPV24}.

Still, within the setting of 2-monads, we have a particular interest in the notion of bundles provided by freely generated categorical structures, including free distributive, completely distributive, and extensive categories over categories~\cite{MRW12, LV24a, LPV24}.

\subsection{Future work}
Since the contexts above are not encompassed by the present work, we aim to extend our results on preservation properties to the general setting of \cite{Luc18a}, which provides a framework that covers all the settings and open problems described above.

\bibliographystyle{plain}
\bibliography{references}

\begin{thebibliography}{10}

\bibitem{BKP89}
R.~Blackwell, G.M. Kelly, and A.J. Power.
\newblock Two-dimensional monad theory.
\newblock {\em J. Pure Appl. Algebra}, 59(1):1--41, 1989.

\bibitem{BJ01}
F.~Borceux and G.~Janelidze.
\newblock {\em Galois Theories}.
\newblock Number~72 in Cambridge Studies in Advanced Mathematics. Cambridge
  University Press, 2001.

\bibitem{Bou14}
J.~Bourke.
\newblock Two-dimensional monadicity.
\newblock {\em Adv. Math.}, 252:708--747, 2014.

\bibitem{BR70}
J.~Bénabou and J.~Roubaud.
\newblock Monades et descente.
\newblock {\em C. R. Acad. Sci. Paris Sér. A-B}, 270:A96--A98, 1970.

\bibitem{CH02}
M.~M. Clementino and D.~Hofmann.
\newblock Triquotient maps via ultrafilter convergence.
\newblock {\em Proc. Amer. Math. Soc.}, 130(11):3423--3431, 2002.

\bibitem{CH04}
M.~M. Clementino and D.~Hofmann.
\newblock Effective descent morphisms in categories of lax algebras.
\newblock {\em Appl. Categ. Structures}, 12(5--6):413--425, 2004.

\bibitem{CJ11}
M.~M. Clementino and G.~Janelidze.
\newblock A note on effective descent morphisms of topological spaces and
  relational algebras.
\newblock {\em Topology Appl.}, 158(17):2431--2436, 2011.

\bibitem{CJ20}
M.M. Clementino and G.~Janelidze.
\newblock Another note on effective descent morphisms of topological spaces and
  relational algebras.
\newblock {\em Topology Appl.}, 273:106961, 2020.

\bibitem{CJ24}
M.M. Clementino and G.~Janelidze.
\newblock Effective descent morphisms of filtered preorders.
\newblock {\em Order}, 2024.

\bibitem{CL23}
M.M. Clementino and F.~Lucatelli~Nunes.
\newblock Lax comma categories of ordered sets.
\newblock {\em Quaest. Math.}, 46(S1):145--159, 2023.

\bibitem{CL24}
M.M. Clementino and F.~Lucatelli~Nunes.
\newblock Lax comma 2-categories and admissible 2-functors.
\newblock {\em Theory Appl. Categ.}, 40(5):180--226, 2024.

\bibitem{CLP24}
M.M. Clementino, F.~Lucatelli~Nunes, and R.~Prezado.
\newblock Lax comma categories: cartesian closedness, extensivity,
  topologicity, and descent.
\newblock {\em Theory Appl. Categ.}, 41(16):516--530, 2024.

\bibitem{CP24}
M.M. Clementino and R.~Prezado.
\newblock Effective descent morphisms of ordered families.
\newblock {\em Quaest. Math.}, 2024.

\bibitem{Gra66}
J.~W. Gray.
\newblock Fibred and cofibred categories.
\newblock In Samuel Eilenberg, David~K. Harrison, Helmut Rörhl, and Saunders
  Mac~Lane, editors, {\em Proceedings of the Conference on Categorical
  Algebra}, pages 21--83. Springer, 1966.

\bibitem{Gra74}
J.~W. Gray.
\newblock {\em Formal category theory: adjointness for {$2$}-categories},
  volume 391 of {\em Lecture Notes in Mathematics}.
\newblock Springer-Verlag, Berlin-New York, 1974.

\bibitem{HSV20}
M.~Huot, Staton S., and M.~Vákár.
\newblock Correctness of automatic differentiation via diffeologies and
  categorical gluing.
\newblock In J.~Goubault{-}Larrecq and B.~König, editors, {\em Foundations of
  Software Science and Computation Structures - 23rd International Conference,
  {FOSSACS} 2020, Held as Part of the European Joint Conferences on Theory and
  Practice of Software, {ETAPS} 2020, Dublin, Ireland, April 25-30, 2020,
  Proceedings}, volume 12077 of {\em Lecture Notes in Computer Science}, pages
  319--338. Springer, 2020.

\bibitem{JSS93}
G.~Janelidze, D~Schumacher, and R.~Street.
\newblock Galois theory in variable categories.
\newblock {\em Appl. Categ. Structures}, 1:103--110, 1993.

\bibitem{JST04}
G.~Janelidze, M.~Sobral, and W.~Tholen.
\newblock Beyond {B}arr {E}xactness: {E}ffective {D}escent {M}orphisms.
\newblock In {\em Categorical foundations}, volume~97 of {\em Encyclopedia
  Math. Appl.}, pages 359--405. Cambridge Univ. Press, Cambridge, 2003.

\bibitem{JT94}
G.~Janelidze and W.~Tholen.
\newblock Facets of descent, {I}.
\newblock {\em Appl. Categ. Structures}, 2:245--281, 1994.

\bibitem{JT97}
G.~Janelidze and W.~Tholen.
\newblock Facets of descent, {II}.
\newblock {\em Appl. Categ. Structures}, 5:229--248, 1997.

\bibitem{JS14}
George Janelidze and Manuela Sobral.
\newblock What are effective descent morphisms of {Priestley} spaces?
\newblock {\em Topology Appl.}, 168:135--143, 2014.

\bibitem{Joh02}
P.~T. Johnstone.
\newblock {\em Sketches of an Elephant}.
\newblock Oxford University Press, 2002.

\bibitem{KL97}
G.~M. Kelly and S.~Lack.
\newblock On property-like structures.
\newblock {\em Theory Appl. Categ.}, 3:213--250, 1997.

\bibitem{KL00}
G.~M. Kelly and S.~Lack.
\newblock On the monadicity of categories with chosen colimits.
\newblock {\em Theory Appl. Categ.}, 7:148--170, 2000.

\bibitem{Koc95}
A.~Kock.
\newblock Monads for which structures are adjoint to units.
\newblock {\em J. Pure Appl. Algebra}, 104:41--59, 1995.

\bibitem{Lac10}
S.~Lack.
\newblock A 2-categories companion.
\newblock In J.~Baez and J.~May, editors, {\em Towards Higher Categories},
  volume 152 of {\em The IMA Volumes in Mathematics and its Applications}.
  Springer, New York, 2010.

\bibitem{Cre99}
I.~Le~Creurer.
\newblock {\em Descent of Internal Categories}.
\newblock PhD thesis, Université Catholique de Louvain, 1999.

\bibitem{Luc16}
F.~Lucatelli~Nunes.
\newblock On biadjoint triangles.
\newblock {\em Theory Appl. Categ.}, 31(9):217--256, 2016.

\bibitem{Luc17}
F.~Lucatelli~Nunes.
\newblock Freely generated $n$-categories, coinserters and presentations of low
  dimensional categories.
\newblock arXiv:1704.04474, 2017.

\bibitem{Luc18b}
F.~Lucatelli~Nunes.
\newblock On lifting of biadjoints and lax algebras.
\newblock {\em Categ. Gen. Algebr. Struct. Appl.}, 9(1):29--58, 2018.

\bibitem{Luc18a}
F.~Lucatelli~Nunes.
\newblock Pseudo-{K}an extensions and descent theory.
\newblock {\em Theory Appl. Categ.}, 33(15):390--444, 2018.

\bibitem{LucTh}
F.~Lucatelli~Nunes.
\newblock {\em Pseudomonads and Descent}.
\newblock PhD thesis, Universidade de Coimbra e Universidade do Porto, 2018.

\bibitem{Luc19}
F.~Lucatelli~Nunes.
\newblock Pseudoalgebras and non-canonical isomorphisms.
\newblock {\em Appl. Categ. Struct.}, 27(1):55--63, 2019.

\bibitem{Luc21}
F.~Lucatelli~Nunes.
\newblock Descent data and absolute {K}an extensions.
\newblock {\em Theory Appl. Categ.}, 37(18):530--561, 2021.

\bibitem{Luc22}
F.~Lucatelli~Nunes.
\newblock Semantic factorization and descent.
\newblock {\em Appl. Categ. Structures}, 30:1393--1433, 2022.

\bibitem{LPS23}
F.~Lucatelli~Nunes, R.~Prezado, and L.~Sousa.
\newblock Cauchy completeness, lax epimorphisms and effective descent for split
  fibrations.
\newblock {\em Bull. Belg. Math. Soc. Simon Stevin}, 30(1):130--139, 2023.

\bibitem{LPV24}
F.~Lucatelli~Nunes, R.~Prezado, and M.~Vákár.
\newblock Free extensivity via distributivity.
\newblock {\em Port. Math.}, 2024.

\bibitem{LS22}
F.~Lucatelli~Nunes and L.~Sousa.
\newblock On lax epimorphisms and the associated factorization.
\newblock {\em J. Pure Appl. Algebra}, 226(12):107--126, 2022.

\bibitem{LV22}
F.~{Lucatelli Nunes} and M.~Vákár.
\newblock Logical relations for partial features and automatic differentiation
  correctness.
\newblock arXiv:2210.08530, 2022.

\bibitem{LV23}
F.~Lucatelli~Nunes and M.~Vákár.
\newblock {CHAD} for expressive total languages.
\newblock {\em Math. Structures Comput. Sci.}, 33(4--5):311–426, 2023.

\bibitem{LV24b}
F.~Lucatelli~Nunes and M.~Vákár.
\newblock Automatic differentiation for {ML}-family languages: Correctness via
  logical relations.
\newblock {\em Math. Structures Comput. Sci.}, pages 1--60, 2024.

\bibitem{LV24a}
F.~Lucatelli~Nunes and M.~Vákár.
\newblock Free doubly-infinitary distributive categories are cartesian closed.
\newblock arXiv:2403.10447v4, 2024.

\bibitem{LV24}
F.~Lucatelli~Nunes and M.~Vákár.
\newblock Monoidal closure of {G}rothendieck constructions via {$ \Sigma
  $}-tractable monoidal structures and {D}ialectica formulas.
\newblock arXiv:2405.07724v2, 2024.

\bibitem{Mar97}
F.~Marmolejo.
\newblock Doctrines whose structure forms a fully faithful adjoint string.
\newblock {\em Theory Appl. Categ.}, 3(2):22--42, 1997.

\bibitem{Mar99}
F.~Marmolejo.
\newblock Distributive laws for pseudomonads.
\newblock {\em Theory Appl. Categ.}, 5(5):91--147, 1999.

\bibitem{MRW12}
F.~Marmolejo, R.~Rosebrugh, and R.~Wood.
\newblock Completely and totally distributive categories {I}.
\newblock {\em J. Pure Appl. Algebra}, 216(8--9):1775--1790, 2012.

\bibitem{MS93}
J.C. Mitchell and A.~Scedrov.
\newblock Notes on sconing and relators.
\newblock In {\em Computer Science Logic. CSL 1992}, volume 702 of {\em Lecture
  Notes in Computer Science}. Springer, Berlin, Heidelberg, 1993.

\bibitem{Pav91}
D.~Pavlović.
\newblock Categorical interpolation: Descent and the beck-chevalley condition
  without direct images.
\newblock In {\em Category Theory}, volume 1488 of {\em Lecture Notes in
  Mathematics}. Springer, Berlin, Heidelberg, 1991.

\bibitem{PCW00}
A.J Power, G.L. Cattani, and G.~Winskel.
\newblock A representation result for free cocompletions.
\newblock {\em J. Pure Appl. Algebra}, 151(3):273--286, 2000.

\bibitem{Pre24}
R.~Prezado.
\newblock On effective descent {$\mathcal V$}-functors and familial descent
  morphisms.
\newblock {\em J. Pure Appl. Algebra}, 228(5):21, 2024.
\newblock Id/No 107597.

\bibitem{PreTh}
R.~Prezado.
\newblock {\em Some aspects of descent theory and applications}.
\newblock PhD thesis, Universidade de Coimbra e Universidade do Porto, 2024.

\bibitem{PL23}
R.~Prezado and F.~Lucatelli~Nunes.
\newblock Descent for internal multicategory functors.
\newblock {\em Appl. Categ. Structures}, 31(11), 2023.

\bibitem{PL24}
R~Prezado and F.~Lucatelli~Nunes.
\newblock Generalized multicategories: change-of-base, embedding, and descent.
\newblock {\em Appl. Categ. Structures}, 2024.
\newblock To appear.

\bibitem{RT94}
J.~Reiterman and W.~Tholen.
\newblock Effective descent maps of topological spaces.
\newblock {\em Topology Appl.}, 57:53--69, 1994.

\bibitem{Sob04}
M.~Sobral.
\newblock Descent for discrete (co)fibrations.
\newblock {\em Appl. Categ. Structures}, 12:527--535, 2004.

\bibitem{Str80}
R.~Street.
\newblock Fibrations in bicategories.
\newblock {\em Cah. Topol. Géom. Différ.}, 21:111--159, 1980.

\bibitem{SW73}
R.~Street and R.F.C Walters.
\newblock The comprehensive factorization of a functor.
\newblock {\em Bull. Amer. Math. Soc.}, 79(5):936--941, 1973.

\bibitem{Str04}
R.H. Street.
\newblock Categorical and combinatorial aspects of descent theory.
\newblock {\em Appl. Categ. Structures}, 12:537--576, 2004.

\bibitem{Wyl71}
O.~Wyler.
\newblock {TOP}-categories and categorical topology.
\newblock {\em General Topology and its Applications}, 1(1):17--28, 1971.

\bibitem{Zöb76}
V.~Zöberlein.
\newblock Doctrines on 2-categories.
\newblock {\em Math. Z.}, 148:267--279, 1976.

\end{thebibliography}

\end{document}